\documentclass[a4paper, 12pt]{article}
\usepackage[english]{babel}          
\usepackage{amsmath, amsfonts, amssymb, amsthm}

\newcounter{num}[section]

\newenvironment{theorem}
{\refstepcounter{num}%
\bigskip\noindent\nopagebreak[4]{\bf Theorem~\arabic{section}.\arabic{num}. }\it}

\newenvironment{lemma}
{\refstepcounter{num}%
\bigskip\noindent\nopagebreak[4]{\bf Lemma~\arabic{section}.\arabic{num}. }\it}

\newenvironment{example}
{\refstepcounter{num}%
\bigskip\noindent\nopagebreak[4]{\bf Example~\arabic{section}.\arabic{num}. }}

\newenvironment{statement}
{\refstepcounter{num}%
\bigskip\noindent\nopagebreak[4]{\bf Statement~\arabic{section}.\arabic{num}. }\it}

\newcommand{\N}{{\mathbb{N}}}

\newcommand{\LL}{{\mathcal{L}}}

\newcommand{\Ss}{{\mathcal{S}}}
\newcommand{\V}{{\mathrm{V}}}

\newcommand{\A}{{\mathcal{A}}}

\newcommand{\B}{{\mathcal{B}}}

\renewcommand{\t}{{\tau}}
\newcommand{\s}{{\sigma}}

\begin{document}

\author{Artem N. Shevlyakov\footnote{The author was supported by Russian Fund of Fundamental Research (project 14-01-00068, the results of Sections 6,7)  and Russian Science Foundation (project 14-11-00085, the results of Section 4,5) }}

\title{New problems in universal algebraic geometry illustrated by boolean equations}

\maketitle

\abstract{We discuss new problems in universal algebraic geometry and explain them by boolean equations}


\section{Introduction}

The process of solving equations is the central part of mathematics. The most general and important problems in this area are the following.
\begin{enumerate}
\item Is a given equation consistent over an algebraic structure (algebra for shortness) $\A$? 
\item Find all solution of a given equation over an algebra $\A$.
\end{enumerate}
There are many surveys and papers devoted to equations in various classes of algebras. Let us just mention about the survey~\cite{romankov_survey} for group equations.

However the recent achievements of universal algebraic geometry (see the papers~\cite{uniTh,uniTh_new,Plot4} by E.Daniyarova, A.Miasnikov, V.Remeslennikov, and B.Plotkin) allow us to pose new problems about equations (all required definitions may be found in Section~\ref{sec:basics} of the current paper).

\begin{enumerate}
\item[3.] {\it Systems of equations VS algebraic sets.} Let $Y$ be an algebraic set over an algebra $\A$. Obviously, there exist more than one systems of equations (systems for shortness) with the solution set $Y$. Let us fix a family of systems $\mathbf{S}$, and let $S(Y)\subseteq\mathbf{S}$ be all systems  with the solution set $Y$. It turns out that the numbers $|S(Y)|$ have a wide spread of values for almost all natural $\mathbf{S}$ (e.g. in~\cite{shevl_random_semilattice_equations} this fact was proved for semilattice equations). Thus, there arises a problem: is there an algebra $\A$ and a natural family $\mathbf{S}$ such that the variance of the set $\{|S(Y)|\}$ is minimal? 

The sense of this problem is the following. Suppose we want to generate random algebraic sets over an algebra $\A$ by a random generation of systems from $\mathbf{S}$. If the variance of the set $\{|S(Y)|\}$ is small,  the random distribution of algebraic sets becomes close to the uniform distribution.

\item[4.] {\it Irreducible algebraic sets.} Let $Y$ be an algebraic set over an algebra $\A$. Is there an algorithm that decides whether $Y$ is irreducible or not? If $Y$ is reducible, can we find its irreducible components? Can you find the average number of irreducible components of all algebraic sets in $\A^n$?

The importance of this problem is the following. According to~\cite{uniTh,uniTh_new}, the structure  of irreducible algebraic sets over $\A$ determines the universal theory of $\A$. Moreover, if $\A$ is finite  irreducible algebraic sets over $\A$ correspond to subalgebras of $\A$. 

\item[5.] {\it Isomorphic algebraic sets.} In~\cite{uniTh} it was defined  isomorphisms between algebraic sets. Namely, isomorphic algebraic sets have the same properties with respect to universal algebraic geometry.  For any algebra $\A$ one can pose the following problem: how many non-isomorphic algebraic sets are there in $\A^n$? The solution of this problem allows us to decide about the complexity of the class of all algebraic sets over $\A$.

\item[6.] {\it Equationally extremal algebras}. Let $\A_n$ be the class of $\LL$-algebras of order $n$ (for example, $\A_n$ is the class of all semilattices of order $n$) and $\mathbf{S}$ a finite set of systems. The problem is the following: find an algebra $\A\in\A_n$ such that the number of consistent systems from $\mathbf{S}$ is maximal (minimal) for $\A$.
\end{enumerate}
Let us refer to the papers, where the problems above were solved for some algebras. In~\cite{shevl_irr_alg_sets_over_LOS} we describe irreducible algebraic sets and compute the average number of irreducible components of algebraic sets over linearly ordered semilattices. In~\cite{shevl_random_semilattice_equations} for the class of semilattices of order $n$ it was described equationally extremal semilattices which have maximal (minimal) number of consistent equations. Above we mentioned about the paper~\cite{shevl_random_semilattice_equations}, where we consider the 3rd problem for semilattices. The obtained results of all papers above show that the problems 3--6 are nontrivial even for simple algebras. However, there exists a class of algebras, the class of boolean algebras, where the problems above have nice solutions.
 
{\it Thus, the aim of this paper is the solution of problems 3--5 in the class of finite boolean algebras} (the sixth problem is unreasonable for boolean algebras, since $|\A_n|\leq 1$ for any $n\in\N$). So the reader may consider this paper as a vast example for problems above.

Let us explain the plan of our paper. In Section~\ref{sec:basics} we give basics notions of universal algebraic geometry. Section~\ref{sec:transformations} contains the rules of transformations of equations over boolean algebras. Actually, any boolean system  $\Ss(X)$ in $n$ variables $X$  can be equivalently reduced to  an {\it orthogonal} system  $\Ss^\prime(Z)$ in $2^n$ variables $Z$. Solving the 3rd problem, we prove that any algebraic set defined by a system in $n$ variables is isomorphic to the solution set of a unique orthogonal system in $2^n$ variables.  

In Section~\ref{sec:irred} we describe irreducible algebraic sets over finite boolean algebras and decompose any algebraic set into a finite union of irreducible ones. In Section~\ref{sec:average_of_irred} we count the average number of irreducible components of algebraic sets over finite boolean algebras. In Section~\ref{sec:rank} we give the definition of a rank of irreducibility $IR(\Ss)$ of a system $\Ss$ and count the average rank of irreducibility of all orthogonal systems in $2^n$ variables. Thus, Sections~\ref{sec:irred}--\ref{sec:rank} solve the 4th problem for finite boolean algebras.

In Section~\ref{sec:isomorphic} we study the 5th problem and directly compute the number of pairs of isomorphic algebraic sets defined by orthogonal systems in $2^n$ variables.

\section{Basic notions}
\label{sec:basics}

Let $\LL=\{\vee^{(2)},\cdot^{(2)},\bar{ }^{\ (1)},0,1\}$ be a language of binary functional symbols $\vee,\cdot$ (join and meet), unary symbol~$\bar{ }$ (complement) and constant symbols $0,1$. Clearly, boolean algebras are algebraic structures of the language $\LL$ with natural interpretation of functional and constant symbols (see~\cite{monk} for more details). 

Recall that for any finite boolean algebra $\B$ there exists a number $r\geq 1$ such that $\B$ is isomorphic to the power set algebra on $r$ elements ($|\B|=2^r$). The number $r$ is called the  {\it rank} of a boolean algebra $\B$. We assume below that any boolean algebra $\B$ is nontrivial, i.e. $|\B|>1$.

An element $a$ is an  {\it atom (co-atom)} if  $\{ab\mid b\in\B\}=\{0,a\}$ (respectively, $\{a\vee b\mid b\in\B\}=\{a,1\}$). Remark that the rank of a finite boolean algebra is equal to the number of atoms (co-atoms).

\bigskip

Following~\cite{uniTh}, let us give the basic notions of algebraic geometry over boolean algebras.

Let $X=\{x_1,x_2,\ldots,x_n\}$ be a finite set of variables. A term $t(X)$ of the language $\LL$ is called an $\LL$-term. The set of all $\LL$-terms in variables $X$ is denoted by $\mathcal{T}_\LL(X)$. A {\it boolean equation}  is an atomic formula $\t(X)=\s(X)$ of the language $\LL$ ($\t,\s$ are $\LL$-terms). The examples of boolean equations are the following expressions: $x_i=x_j$, $x_1=1$, $x_1x_2=x_3\vee x_4$, $\bar{x}_1\vee x_2=\overline{x_3x_4}$.

A {\it system of equations}  ({\it system} for shortness) is an arbitrary set of boolean equations. The {\it set of all solutions (solution set)} of a system $\Ss$ over a boolean algebra  $\B$ is denoted by  $\V_\B(\Ss)$.

A set $Y\subseteq \B^n$ is ~\textit{algebraic} over a boolean algebra $\B$ if there exists a system  $\Ss$ such that $Y=\V_\B(\Ss)$. A nonempty algebraic set  $Y$ is {\it irreducible} if it is not a finite proper union of other algebraic sets. According to~\cite{uniTh}, it follows that each algebraic set $Y\subseteq \B^n$ is decomposable into a finite union of irreducible algebraic sets
\begin{equation}
Y=Y_1\cup Y_2\cup\ldots\cup Y_m\; (Y_i\nsubseteq Y_j\mbox{ for }i\neq j),
\label{eq:Y_irred_components}
\end{equation}
and the decomposition~(\ref{eq:Y_irred_components}) is unique up to the permutation of the sets $Y_i$. The sets $Y_i$ in~(\ref{eq:Y_irred_components}) are called the {\it irreducible components} of a set $Y$.

Let $Y=\V_\B(\Ss)$ be an algebraic set over a boolean algebra $\B$, and $\Ss$ depends on variables $X=\{x_1,x_2,\ldots,x_n\}$. One can define an equivalence relation $\sim_Y$ on $\mathcal{T}_\LL(X)$ as follows:  
\[
t(X)\sim_Y s(X)\Leftrightarrow t(P)=s(P) \mbox{ for each point $P\in Y$}.
\] 
The set of $\sim_Y$-equivalence classes is called the  {\it coordinate algebra} of  $Y$ and denoted by  $\Gamma_\B(Y)$. By the results of~\cite{uniTh}, it follows that $\Gamma_\B(Y)$ is a boolean algebra and generated by the elements  $x_1,x_2,\ldots,x_n$. In other words, all coordinate algebras are finitely generated, and, therefore,  {\it all coordinate algebras of algebraic sets over boolean algebras are finite}. The following statement describes the properties of coordinate algebras of irreducible algebraic sets.

\begin{theorem}
An algebraic set $Y$ is irreducible over a boolean algebra  $\B$ iff $\Gamma_\B(Y)$ is embedded into $\B$ 
\label{th:gamma_is_embedded_for_irr}
\end{theorem}  
\begin{proof}
Actually, in~\cite{uniTh} (Theorem A) it was proved that  $\Gamma_\B(Y)$ is discriminated by  $\B$ iff the algebraic set $Y$ is irreducible. Since $\Gamma_\B(Y)$ is finite, the discrimination is equivalent to the embedding of  $\Gamma_\B(Y)$ into $\B$.
\end{proof}

There are different algebraic sets with isomorphic coordinate algebras. For example, the following algebraic sets
\[
Y_1=\V_\B(\{x_1x_2=x_2\}),\; Y_2=\V_\B(\{{x}_1{x}_2=x_1\})
\]
have isomorphic coordinate algebras  $\Gamma_\B(Y_1),\Gamma_\B(Y_2)$, since the second equation above is obtained from the first one by the variable substitution (in Example~\ref{ex:first} we directly compute the coordinate algebra of the sets $Y_1,Y_2$).
 
Following~\cite{uniTh}, an algebraic sets are \textit{isomorphic} if they have  isomorphic coordinate algebras.

\begin{example}
\label{ex:first}
Let us compute the coordinate algebra of the algebraic set $Y=\V_\B(x_1x_2=x_2)$, where $\B$ is an arbitrary nontrivial boolean algebra. By the definition, $\Gamma_\B(Y)$ is generated by the elements $x_1,x_2$ (we identify here a term $x_i$ with its $\sim_Y$-equivalence class). According to the axioms of boolean algebras, the equality  $x_1x_2=x_2$ gives that the term $\bar{x}_1 x_2$ equals $0$ in $\Gamma_\B(Y)$. The direct computations give that  $\Gamma_\B(Y)$ consists of $8$ elements ($\sim_Y$-equivalence classes)
\[
0,1,x_1,x_2,x_1\bar{x_2},\bar{x}_1,\bar{x}_2,x_2\vee\bar{x}_1.
\] 
Therefore,  $\Gamma_\B(Y)$ is isomorphic to a boolean algebra of rank $3$, and the elements $x_2$, $x_1\bar{x}_2$, $\bar{x}_1$ ($x_1$, $\bar{x}_1\vee x_2$, $\bar{x}_2$) are atoms  (respectively, co-atoms) of $\Gamma_\B(Y)$).

\end{example}

\section{Transformations of boolean equations}
\label{sec:transformations}

Let $X=\{x_1,x_2,\ldots,x_n\}$ be a finite set of variables. Let us define new variables  $Z=\{z_\alpha|\alpha\in\{0,1\}^n\}$ indexed by all $n$-tuples $\alpha\in\{0,1\}^n$ ($|Z|=2^n$). Following~\cite{rudeanu}, the variables $Z$ are called \textit{orthogonal}. By $\pi_i(\alpha)$ ($1\leq i\leq n$) we denote the projection of a tuple $\alpha\in\{0,1\}^n$ onto the $i$-th coordinate. The substitution of the variables $X=\{x_1,x_2,\ldots,x_n\}$ is the following
\begin{equation}
\label{eq:old_var_by_new}
x_i=\bigvee_{\pi_i(\alpha)=1}z_\alpha.
\end{equation}
For example, the set $X=\{x_1,x_2\}$ gives $Z=\{z_{(0,0)},z_{(0,1)},z_{(1,0)},z_{(1,1)}\}$ and 
\[
x_1=z_{(1,0)}\vee z_{(1,1)},\; x_2=z_{(0,1)}\vee z_{(1,1)}.
\]
According to the axioms of boolean algebras, it follows that the variables $Z$ are obtained from $X$ by the following rules:
\begin{equation}
\label{eq:new_var_by_old}
z_\alpha=x_1^{a_1}x_2^{a_2}\ldots x_n^{a_2},
\end{equation} 
where $\alpha=(a_1,a_2,\ldots,a_n)$, $a_i\in\{0,1\}$ and
\begin{equation}
x_i^{a_i}=\begin{cases}
x_i \mbox{ if } a_i=1,\\
\overline{x}_i \mbox{ if } a_i=0.
\end{cases}
\end{equation} 
For example, the sets $X=\{x_1,x_2\}$, $Z=\{z_{(0,0)},z_{(0,1)},z_{(1,0)},z_{(1,1)}\}$ give $z_{(0,0)}=\overline{x}_1\overline{x}_2$, $z_{(0,1)}=\overline{x}_1x_2$, $z_{(1,0)}={x}_1\overline{x}_2$, $z_{(1,1)}={x}_1x_2$. 

By~(\ref{eq:old_var_by_new}), any system $\Ss^\prime$ in variables $X=\{x_1,x_2,\ldots,x_n\}$ can be written as 
\begin{equation}
\Ss=\{z_\alpha=0\mid \alpha\in A\}\cup
\label{eq:S(Z)}
\bigcup_{\substack{ \alpha\neq\beta}}\{z_\alpha z_\beta=0\}\cup
\{\bigvee_{\alpha}z_\alpha=1\}, 
\end{equation}
where $A\subseteq\{0,1\}^n$ and $\bigvee_{\alpha}z_\alpha$ is the join of all variables $z_\alpha\in Z$ (see~\cite{shevl_boolean} for more details). 

Moreover, in~\cite{shevl_boolean} it was proved that the algebraic sets  $\V_\B(\Ss^\prime)$, $\V_\B(\Ss)$ are isomorphic. A system of the form~(\ref{eq:S(Z)}) is called  \textit{orthogonal}.

\begin{example}
The set $Y=\V_\B(x_1x_2=x_2)$ ($\B$ is an arbitrary boolean algebra) is isomorphic to the solution  set of a system
\begin{equation}
\label{eq:from_example}
\begin{cases}
z_{(0,1)}=0,\\
z_{(0,0)}z_{(0,1)}=z_{(0,0)}z_{(1,0)}=z_{(0,0)}z_{(1,1)}=z_{(0,1)}z_{(1,0)}=z_{(0,1)}z_{(1,1)}=
z_{(1,0)}z_{(1,1)}=0,\\
z_{(0,0)}\vee z_{(0,1)}\vee z_{(1,0)}\vee z_{(1,1)}=1
\end{cases}
\end{equation}
since
\[
x_1x_2=x_2\Leftrightarrow \bar{x}_1{x}_2=0\Leftrightarrow z_{(0,1)}=0.
\]
\end{example} 

\begin{statement}
The coordinate algebra of the solution set of an orthogonal system $\Ss$~(\ref{eq:S(Z)}) is isomorphic to the boolean algebra of rank $m-a$, where $m=|Z|=2^n$ and $a=|A|$. 
\label{st:coord_alg_of_ort_syst}
\end{statement}
\begin{proof}
Since all points $P_\alpha=(p_\beta\mid \beta\in\{0,1\}^n)$ ($\alpha\notin A$),
\[
p_\beta=\begin{cases}
1\mbox{ if }\beta=\alpha,\\
0\mbox{ otherwise}
\end{cases}
\] 
belong to  $Y=\V_\B(\Ss)$, the definition of the $\sim_Y$-equivalence gives that the elements $z_\alpha$ ($\alpha\notin A$) are nonzero in $\Gamma_\B(Y)$ and $z_\alpha\neq z_{\alpha^\prime}$ for distinct $\alpha,\alpha^\prime\notin A$). The equations $z_\alpha z_\beta=0\in\Ss$ imply that the elements $z_\alpha$ ($\alpha\notin A$) are exactly the atoms of the boolean algebra $\Gamma_\B(Y)$. Since the rank of a boolean algebra is equal to the number of atoms, $\Gamma_\B(Y)$ is isomorphic to the boolean algebra of rank  $m-a$.  
\end{proof}

\begin{example}
\label{ex:second}
According to Statement~\ref{st:coord_alg_of_ort_syst}, the coordinate algebra of the solution set of an orthogonal system $\Ss$~(\ref{eq:from_example}) is isomorphic to the boolean algebra of rank $3$ (in Example~\ref{ex:first} we directly obtained the same result). Using Theorem~\ref{th:gamma_is_embedded_for_irr}, we obtain that the set  $Y$ is irreducible over any boolean algebra of rank $r\geq 3$. 

If $\B$ is the boolean algebra of rank $2$ the solution set of~(\ref{eq:from_example}) is decomposable into the union of solution sets of the following systems 
\[
\Ss_1=\Ss\cup \{z_{(0,0)}=0\},\;\Ss_2=\Ss\cup \{z_{(1,0)}=0\},\;\Ss_3=\Ss\cup \{z_{(1,1)}=0\}.
\]
For the boolean algebra $\B$ of rank $2$ there are not nonzero elements $z_1,z_2,z_3\in\B$ with $z_iz_j=0$ ($i\neq j$). Therefore, for any solution of $\Ss$ one of the following equalities holds $z_{(0,0)}=0$, $z_{(1,0)}=0$, $z_{(1,1)}=0$. Thus,  $\V_\B(\Ss)$ can be decomposed into a union of solution sets of  $\Ss_1,\Ss_2,\Ss_3$.  
\end{example}

\bigskip

One can prove that for any algebraic set $Y\subseteq\B^n$ there exists a unique orthogonal system $\Ss$ in $m=2^n$ variables with the solution set isomorphic to  $Y$. Therefore, there arises a one-to-one correspondence between algebraic sets in $\B^n$ and orthogonal systems in $m=2^n$ variables. It allows us below to identify the class of algebraic sets in $\B^n$ and the class of all orthogonal systems in $m=2^n$ variables.

\section{Irreducible components of algebraic sets}
\label{sec:irred}
Let $Y$ be the solution set of $\Ss$~(\ref{eq:S(Z)}) over the boolean algebra $\B$ of rank $r$. Let  $m=|Z|=2^n$, $a=|A|$. 

\begin{lemma}
If $m-a\leq r$, then $Y$ is irreducible.
\label{l:when_Y_is_irred}
\end{lemma} 
\begin{proof}
It directly follows from Statement~\ref{st:coord_alg_of_ort_syst} and Theorem~\ref{th:gamma_is_embedded_for_irr}.
\end{proof}

\begin{lemma}
Let $m-a>r$ then $Y$ is a union of solution sets of the following orthogonal systems
\begin{equation}
\Ss_{B}=\{z_\alpha=0\mid \alpha\in B\}\cup
\label{eq:S_B}
\bigcup_{\substack{ \alpha\neq\beta}}\{z_\alpha z_\beta=0\}\cup
\{\bigvee_{\alpha}z_\alpha=1\}
\end{equation}
where $B\subseteq\{0,1\}^n$, $B\supseteq A$, $|B|=m-r$.
Moreover, the sets  $Y_B=\V_\B(\Ss_B)$ are irreducible components of  $Y$.
\label{l:when_Y_is_red}
\end{lemma}
\begin{proof}
Actually, the statement of this lemma was demonstrated in Example~\ref{ex:second}, where the solution set of $\Ss$ over the boolean algebra of rank $2$ is a union of the solution sets of the systems $\Ss_1,\Ss_2,\Ss_3$. For the systems $\Ss_1,\Ss_2,\Ss_3$ the set $B$ respectively equals $\{(0,0),(0,1)\}$,  $\{(1,0),(0,1)\}$, $\{(1,1),(0,1)\}$.

\bigskip

The proof of the lemma follows from the statements below.

\begin{enumerate}
\item {\it Let us prove $Y=\bigcup_{B} Y_B$}. Since the systems $\Ss_B$ contain new equalities  $z_\alpha=0$, then obviously $\V_\B(\Ss_B)\subseteq \V_\B(\Ss)$ and $\bigcup_{B} Y_B\subseteq Y$.

Let us prove the inverse inclusion $Y\subseteq \bigcup Y_B$. Let  $P=(p_\alpha\mid \alpha\in\{0,1\}^n)\in Y$. Since $p_\alpha p_\beta=0$ for all  $\alpha\neq \beta$, then $P$ contains at most $r$ nonzero coordinates (and at least $m-r$ zero coordinates). Therefore, there exists a set  $B\subseteq \{0,1\}^n$, $|B|=m-r$, $B\supseteq A$ such that $p_\beta=0$ for all  indexes $\beta\in B$, and therefore $P\in Y_{B}$. 

\item Statement~\ref{st:coord_alg_of_ort_syst} implies that the coordinate algebras of algebraic sets $Y_B$ are isomorphic to $\B$. By Theorem~\ref{th:gamma_is_embedded_for_irr}, {\it all sets $Y_B$ are irreducible}.

\item Let us prove that {\it $Y_B\nsubseteq Y_{B^\prime}$ for distinct sets $B,B^\prime$}. Let $\beta\in B\setminus B^\prime$. Then the point $P=(p_\alpha\mid \alpha\in\{0,1\}^n)$ with coordinates
\[
p_\alpha=\begin{cases}
1\mbox{ if }\alpha=\beta\\
0\mbox{ otherwise }
\end{cases}
\]
belongs to $Y_{B^\prime}$, but $P\notin Y_B$.  
\end{enumerate}

\end{proof}

\section{Average number of irreducible components}
\label{sec:average_of_irred}
In this section we obtain a formula for the average number of irreducible components of algebraic sets defined by orthogonal systems~(\ref{eq:S(Z)}) over the boolean algebra  $\B$ of rank $r$. Let $m$ be the number of variables in the orthogonal system $\Ss$~(\ref{eq:S(Z)}) and $a=|A|$. According to Lemmas~\ref{l:when_Y_is_irred},~\ref{l:when_Y_is_red}, the number of irreducible components $\mathrm{Irr}(\Ss)$ of the solution set of $\Ss$ equals
\[
\mathrm{Irr}(\Ss)=\begin{cases}
1\mbox{ if }m-a\leq r\\
\binom{m-a}{r}\mbox{ otherwise }
\end{cases}
\]

The number of orthogonal systems for fixed $m,a$ equals $\binom{m}{a}$ 
The number of all orthogonal systems is $2^m$, therefore the average number of irreducible components of algebraic sets defined by orthogonal systems in $m$ variables equals
\[
\overline{\mathrm{Irr}}=
\frac{1}{2^m}\left({1\cdot\sum_{a=m-r}^{m}\binom{m}{a}+\sum_{a=0}^{m-r-1}\binom{m}{a}\binom{m-a}{r}}\right).
\]
We have
\begin{multline*}
\sum_{a=0}^{m-r-1}\binom{m}{a}\binom{m-a}{r}=\sum_{a=0}^{m-r-1}\binom{m}{m-a}\binom{m-a}{r}=\\
\binom{m}{r}\sum_{a=0}^{m-r-1}\binom{m-r}{m-r-a}=\binom{m}{r}\sum_{a=0}^{m-r-1}\binom{m-r}{a}=
\binom{m}{r}(2^{m-r}-1),
\end{multline*}
and, therefore, the average number of irreducible components is 
\begin{multline*}
\overline{\mathrm{Irr}}=
\frac{1}{2^m}\left({\sum_{a=m-r}^{m}\binom{m}{a}+\binom{m}{r}(2^{m-r}-1)}\right)=
\frac{1}{2^m}\left({\sum_{a=m-r+1}^{m}\binom{m}{a}+2^{m-r}\binom{m}{r}}\right)=\\
\frac{1}{2^m}\left({\sum_{i=0}^{r-1}\binom{m}{i}+2^{m-r}\binom{m}{r}}\right).
\end{multline*}
For a fixed $r$ and $m\to \infty$ we have $\frac{1}{2^m}\sum_{i=0}^{r-1}\binom{m}{i}\to 0$ and
\[
\overline{\mathrm{Irr}}\sim 2^{-r}\binom{m}{r}\mbox{ for }m\to \infty.
\]

\section{Ranks of irreducibility}
\label{sec:rank}
According to Lemmas~\ref{l:when_Y_is_irred},~\ref{l:when_Y_is_red}, the solution set of a system $\Ss$~(\ref{eq:S(Z)}) may be reducible over the boolean algebra of rank $r$, but the solution set of $\Ss$ becomes irreducible over the boolean algebras of higher ranks. We say that a system $\Ss$~(\ref{eq:S(Z)}) has  {\it the rank of irreducibility} $IR(\Ss)$ if  the solution set of $\Ss$ is irreducible over the boolean algebra of rank  $IR(\Ss)$, but solution set of $\Ss$ is reducible over each boolean algebra of rank $r<IR(\Ss)$ (if  $\Ss$ is inconsistent over any boolean algebra we put $IR(\Ss)=0$). Below we compute the average rank of irreducibility of orthogonal systems in $m$ variables. 

By Lemmas~\ref{l:when_Y_is_irred},~\ref{l:when_Y_is_red}, we have that $IR(\Ss)$ of a system $\Ss$~(\ref{eq:S(Z)}) equals $m-a$, $a=|A|$. The number of orthogonal systems in $m$ variables with the rank of irreducibility $m-a$ is equal to $\binom{m}{a}$. Therefore, the average rank of irreducibility of orthogonal systems in $m$ variables is
\[
2^{-m}\sum_{a=0}^m(m-a)\binom{m}{a}=2^{-m}\left(m\sum_{a=0}^m \binom{m}{a}-\sum_{a=0}^m a\binom{m}{a}\right)=2^{-m}\left(m2^m-m2^{m-1}\right)=m/2.
\]

\section{Pairs of isomorphic algebraic sets}
\label{sec:isomorphic}
In this section we compute the number of pairs $(Y_1,Y_2)$ such that the algebraic sets $Y_i$ are isomorphic to each other and $Y_i$ are defined by orthogonal systems in $m$ variables.

Suppose algebraic sets  $Y_i$ are defined by the following orthogonal systems
\begin{equation}
\Ss_i=\{z_\alpha=0\mid \alpha\in A_i\}\cup
\bigcup_{\substack{ \alpha\neq\beta}}\{z_\alpha z_\beta=0\}\cup
\{\bigvee_{\alpha}z_\alpha=1\},
\label{eq:S_i}
\end{equation}
where $A_i\subseteq\{0,1\}^n$

The following lemma is a simple corollary of Statement~\ref{st:coord_alg_of_ort_syst}.

\begin{lemma}
Algebraic sets $Y_1,Y_2$ defined by orthogonal systems $\Ss_1,\Ss_2$~(\ref{eq:S_i}) are isomorphic to each other iff  $|A_1|=|A_2|$.
\label{l:when_two_AS_isomorphic}
\end{lemma}
\begin{proof}
In~\cite{uniTh} (Corollary 5.7) it was proved that that $Y_1,Y_2$ are isomorphic iff their coordinate algebras are the isomorphic. The application of Statement~\ref{st:coord_alg_of_ort_syst} concludes the proof. 
\end{proof}

\medskip

The number of pairs $(\Ss_1,\Ss_2)$ with $|A_1|=|A_2|$ is equal to 
\[
\sum_{i=0}^m\binom{m}{i}\binom{m}{i}=\binom{2m}{m}.
\]

Since there are exactly $2^m2^m=4^m$ pairs of algebraic sets defined by orthogonal systems in $m$ variables, two random algebraic sets are isomorphic with the following probability
\[
\frac{\binom{2m}{m}}{2^m\cdot 2^m}.
\]
Applying Stirling formula to the expression $\binom{2m}{m}$, we obtain that the required probability asymptotically equals $\frac{1}{\sqrt{\pi m}}$.

The information of the author:

Artem N. Shevlyakov

Sobolev Institute of Mathematics

644099 Russia, Omsk, Pevtsova st. 13

Phone: +7-3812-23-25-51.

\bigskip

Omsk State Technical University

644050 Russia, Omsk, pr. Mira, 11

\bigskip

e-mail: \texttt{a\_shevl@mail.ru}
\end{document}